\documentclass{amsart}

\title{Rational cuspidal curves with four cusps on Hirzebruch surfaces}
\author{Torgunn Karoline Moe}
\address{Department of Mathematics, University of Oslo, P.O. Box 1053 Blindern, NO-0316 Oslo, NORWAY}
\email{torgunnk@math.uio.no}



\date{\today}

\usepackage{mathrsfs}
\usepackage{graphicx}
\usepackage{hyperref}
\usepackage{amssymb,amsfonts,amssymb,amsthm}
\usepackage{moreverb}
\usepackage{subfigure}
\usepackage{multirow}
\usepackage{array,float}
\usepackage{enumerate}

\usepackage{booktabs}

\input xy
\xyoption{all}

\theoremstyle{definition}

\theoremstyle{plain}
\newtheorem{thm}{Theorem}[section]
\newtheorem{lem}[thm]{Lemma}
\newtheorem{prp}[thm]{Proposition}
\newtheorem{cor}[thm]{Corollary}
\newtheorem*{cor*}{Corollary}
\newtheorem{conj}[thm]{Conjecture}

\newtheorem*{conj*}{Conjecture}
\newtheorem*{thm*}{Theorem}
\newtheorem*{prb*}{Problem}
\newtheorem*{qe*}{Question}

\theoremstyle{definition}

\newtheorem{ex}[thm]{Example}

\newtheorem{rem}[thm]{Remark}

\theoremstyle{remark}

\numberwithin{equation}{section}

\newcommand{\Po}{\ensuremath{\mathbb{P}^2}}
\newcommand{\PP}{\ensuremath{\mathbb{P}^1 \times \mathbb{P}^1}}
\newcommand{\Fe}{\ensuremath{\mathbb{F}_e}}
\newcommand{\Fh}{\ensuremath{\mathbb{F}_h}}
\newcommand{\Fen}{\ensuremath{\mathbb{F}_1}}
\newcommand{\Pot}{\ensuremath{\mathbb{P}^1}}

\newcommand{\C}{\ensuremath{\mathbb{C}}}

\newcommand{\h}{\mathrm{h}}
\newcommand{\Z}{\ensuremath{\mathbb{Z}}}

\newcommand{\Q}{\ensuremath{\mathbb{Q}}}
\newcommand{\R}{\ensuremath{\mathbb{R}}}
\newcommand{\V}{\ensuremath{\mathscr{V}}}

\newcommand{\RH}{M}
\newcommand{\LH}{L}

\newcommand{\Os}{\ensuremath{\mathscr{O}}}
\newcommand{\noi}{\noindent}
\newcommand{\ol}{\overline}

\newcommand{\beq}{\begin{equation*}}
\newcommand{\eeq}{\end{equation*}}
\newcommand{\bsp}{\begin{split}}
\newcommand{\esp}{\end{split}}

\newcommand{\T}{\ensuremath{\Theta_V \left\langle D \right\rangle}}
\numberwithin{equation}{section}
\newcommand{\lk}{\ol{\kappa}(\Po \setminus C)}

\newcommand{\lkkf}{\ol{\kappa}(\Fe \setminus C)}

\newcommand{\tr}{\toprule}
\newcommand{\br}{\bottomrule}

\newcommand{\ldot}{\ensuremath{\,.\,}}
\newcommand{\hide}[1]{}
\begin{document}

\begin{abstract}
The purpose of this article is to shed light on the question of how many and what kind of cusps a rational cuspidal curve on a Hirzebruch surface can have. We use birational transformations to construct rational cuspidal curves with four cusps on the Hirzebruch surfaces and find associated results for these curves. 
\end{abstract}

\maketitle

\setcounter{tocdepth}{1}

\tableofcontents

\section{Introduction}

Let $C$ be a reduced and irreducible curve on a smooth complex surface $X$. A point $p$ on $C$ is called a \emph{cusp} if it is singular and if the germ $(C,p)$ of $C$ at $p$ is irreducible. A curve $C$ is called \emph{cuspidal} if all of its singularities are cusps, and $j$-cuspidal if it has $j$ cusps $p_1, \ldots, p_j$. Let $g$ denote the geometric genus of the curve, and recall that the curve is called \emph{rational} if $g=0$.

Assume that two curves $C$ and $C'$, without common components, meet at a point $p$, then $p$ is called an \emph{intersection point} of $C$ and $C'$. If $f$ and $f'$ are local equations of $C$ and $C'$ at $p$, then by the \emph{intersection multiplicity} $(C \cdot C')_p$ of $C$ and $C'$ at $p$ we mean $$(C \cdot C')_p := \dim_{\C}\Os_{X,p}/(f,f').$$ 

We sometimes view the intersection of $C$ and $C'$ as a $0$-cycle, and express this by the notation $C \cdot C'$, where $$C \cdot C' = \sum _{p \in C \cap C'}(C \cdot C')_pp.$$

For any two divisors $C$ and $C'$ on $X$, we calculate the \emph{intersection number} $C \ldot C'$ using linear equivalence and the pairing ${\mathrm{Pic}(X) \times \mathrm{Pic}(X) \rightarrow \Z}$ \cite[Theorem V 1.1, pp.357--358]{Hart:1977}.

By \cite[Theorem V 3.9, p.391]{Hart:1977}, there exists for any curve $C$ on a surface $X$ a sequence of $t$ monoidal transformations, $$V=V_t \xrightarrow{\sigma_t} V_{t-1} \xrightarrow{} \cdots \xrightarrow{} V_1 \xrightarrow{\sigma_1} V_0=X,$$ such that the reduced total inverse image of $C$ under the composition $\sigma:V \rightarrow X$, $$D := \sigma^{-1}(C)_{\mathrm{red}},$$ is a \emph{simple normal crossing divisor} (SNC-divisor) on the smooth complete surface $V$ (see \cite{Iitaka}). The pair $(V,D)$ and the transformation $\sigma$ are referred to as an \emph{embedded resolution} of $C$, and it is called a \emph{minimal embedded resolution} of $C$ when $t$ is the smallest integer such that $D$ is an SNC-divisor.

Let $p$ be a cusp on a curve $C$, let $m$ denote the \emph{multiplicity} of $p$, and let $m_i$ denote the multiplicity of the infinitely near points $p_i$ of $p$. Then the \emph{multiplicity sequence} $\ol{m}$ of the cusp $p$ is defined to be the sequence of integers $$\ol{m}=[m,m_1,\ldots,m_{t-1}],$$ where $t$ is the number of monoidal tranformations in the local minimal embedded resolution of the cusp, and we have $m_{t-1}=1$ (see \cite{Brieskorn}). We follow the convention of compacting the notation by omitting the number of ending 1s in the sequence and indexing the number of repeated elements, for example, we write $$[6,6,3,3,3,2,1,1]=[6_2,3_3,2].$$ Note that there are further relations between the elements in the sequence (see \cite{FlZa95}).

The collection of multiplicity sequences of a cuspidal curve will be referred to as its \emph{cuspidal configuration}.

The question of how many and what kind of cusps a rational cuspidal curve on a given surface can have naturally arises, and in this article we give new answers to the last part of this question for rational cuspidal curves on Hirzebruch surfaces. The first part of the question is addressed in a subsequent article \cite{MOEONC}.

Our main results can be summarized in the following theorem.
\begin{thm*}
For all $e\geq0$ and $h \in \{0,1\}$, except the pairs $(e,k)=(h,k)=(0,0)$, the following rational cuspidal curves with four cusps exist on $\Fe$ and $\Fh$.
\begin{table}[H]
  \renewcommand\thesubtable{}
  \setlength{\extrarowheight}{2pt}
\centering
	{\noindent \begin{tabular*}{1\linewidth}{@{}c@{\extracolsep{\fill}}l@{}l@{}c}
	\tr
	{\textnormal{\textbf{Type}}} & {\textnormal{\textbf{Cuspidal configuration}}} & \textnormal{\textbf{For}}  & \textnormal{\textbf{Surface}}\\
\tr
	{$(2k+1,4)$}  & $[4_{k-1+e},2_3],[2],[2],[2]$& $k \geq 0$ & $\Fe$\\
 	{$(3k+1-h,5)$}  & $[4_{2k-1+h},2_3],[2],[2],[2]$& $k \geq 0$ & $\Fh$ \\
	{$(2k+2-h,4)$} & $[3_{2k-1+h},2],[2_3],[2],[2]$ & $k \geq 0$ &$\Fh$ \\
	{$(k+1-h,3)$}  & $[2_{n_1}],[2_{n_2}],[2_{n_3}],[2_{n_4}]$& $k \geq 2$ and $\sum n_i = 2k+h$ & $\Fh$\\
	{$(0,3)$}  & $[2],[2],[2],[2]$&  & $\mathbb{F}_2$\\
	\br
	\end{tabular*}}
	\label{tab:fourcusps}
	\end{table}
\end{thm*}

The lack of examples of curves with more than four cusps leads us to the following conjecture.
\begin{conj*}
A rational cuspidal curve on a Hirzebruch surface has at most four cusps.
\end{conj*}

Associated to these curves, we find results regarding the Euler characteristic of the logarithmic tangent sheaf $\chi(\T)$ and the maximal multiplicity of a cusp on a rational cuspidal curve. Moreover, we give an example of a real rational cuspidal curve with four real cusps on a Hirzebruch surface.

\subsection{Structure}
In Section \ref{MN} we motivate the study of rational cuspidal curves on Hirzebruch surfaces by recalling the history of the study of such curves on the projective plane. In Section \ref{PR} we give the basic definitions and preliminary results for (rational) cuspidal curves on Hirzebruch surfaces. Section \ref{4C} contains the main results of this article. Here we construct several series of rational cuspidal curves with four cusps on Hirzebruch surfaces. In Section \ref{AR} we present associated results, and in particular we compute $\chi(\T)$ for the curves we have found in Section \ref{4C}. Moreover, we find a lower bound for the multiplicity of the cusp with the highest multiplicity. Additionally, we construct a real rational cuspidal curve with four real cusps. In the appendix we construct some of the rational cuspidal curves with four cusps on Hirzebruch surfaces that were found in Section \ref{4C}.

\subsection{Acknowledgements}
This article consists of results from my PhD-thesis \cite{MOEPHD}, and it is the second of two articles (see \cite{MOEONC}). I am very grateful to Professor Ragni Piene for suggesting cuspidal curves on Hirzebruch surfaces as the topic of my thesis, and for all the help along the way. Moreover, I am indebted to Georg Muntingh for helping me create images, and to Nikolay Qviller for guiding me over some of the obstacles that I met in this work. Furthermore, I would like to thank Torsten Fenske for sending me a copy of his PhD-thesis, and Professor Hubert Flenner and Professor Mikhail Zaidenberg for valuable comments and suggestions.

\section{Motivation: the case of plane curves}\label{MN}
Let $\Po$ denote the projective plane with coordinates $(x:y:z)$ and coordinate ring $\C[x,y,z]$. A reduced and irreducible curve $C$ on $\Po$ is given as the zero set $\V(F)$ of a homogeneous, reduced and irreducible polynomial $F(x,y,z) \in \C[x,y,z]_d$ for some $d$. In this case, the polynomial $F$ and the curve $C$ is said to have degree $d$. Let $q$ be a smooth point on a plane curve $C$ with tangent $T$. Recall that the point $q$ is called an \emph{inflection point} if the local intersection multiplicity $(C \cdot T)_q$ of $C$ and $T$ at $q$ is $\geq 3$.

Plane rational cuspidal curves have been studied quite intensively both classically and the last 20 years. Classically, the study was part of the process of classifying plane curves, and additionally bounds on the number of cusps were produced (see \cite{Clebsch, Lefschetz, Salmon, Telling, Veronese, Wieleitner}). In the modern context, rational cuspidal curves with many cusps play an important role in the study of open surfaces (see \cite{Fent, FlZa94, Wak}), and the study of these curves was further motivated in the mid 1990s by Sakai in \cite{Sakai}. The two tasks at hand were first to classify all rational and elliptic cuspidal plane curves, and second to find the maximal number of cusps on a rational cuspidal plane curve. 

The first complete classification of rational cuspidal curves of degree $d \leq 5$ that we have found, up to cuspidal configuration, is by Namba in \cite{Namba}, and a list of the curves of degree $d=5$ can be found in Table \ref{tab:degree5} \cite[Theorem 2.3.10, pp.179--182]{Namba}. The subclassification of the curves is due to the appearance of inflection points. The remarkable thing about this list of curves is that there are fewer curves with many cusps than expected. In fact, for curves of degree $d \leq 5$ there are only three curves with three cusps and only one with four cusps. For curves in higher degrees, less is known, but there is a classification of curves of degree $d=6$ by Fenske in \cite{Fent}, and in that case the highest number of cusps is three, and up to cuspidal configuration there are only two such curves. 

\begin{table}[htb]
  \renewcommand\thesubtable{}
  \setlength{\extrarowheight}{2pt}
\centering
{\small
	{\noindent \begin{tabular*}{\linewidth}{@{}c@{\extracolsep{\fill}}l@{}c}
	\toprule
	{\textbf{Curve}} &{\textbf{Cuspidal conf.}}& {\textbf{Parametrization}}\\ 
	\toprule
	$C_{1A}$&$\quad[4]$& $(s^5:st^4:t^5)$\\
	$C_{1B}$&$\quad[4]$&$(s^5-s^4t:st^4:t^5)$\\
	$C_{1C}$&$\quad[4]$&$(s^5+as^4t-(1+a)s^2t^3:st^4:t^5), \; a\neq-1$\\
	$C_2$&$\quad [2_6]$&$(s^4t:s^2t^3-s^5:t^5-2s^3t^2)$\\
	\midrule
	$C_{3A}$&$\quad [3,2],[2_2]$ &$(s^5:s^3t^2:t^5)$\\
	$C_{3B}$&$\quad [3,2],[2_2]$ &$(s^5:s^3t^2:st^4+t^5)$\\
	$C_4$&$\quad [3], [2_3]$&$(s^4t-\frac{1}{2}s^5:s^3t^2:\frac{1}{2}st^4+t^5)$\\
	$C_5$&$\quad [2_4],[2_2]$&$(s^4t-s^5:s^2t^3-\frac{5}{32}s^5:-\frac{47}{128}s^5+\frac{11}{16}s^3t^2+st^4+t^5)$\\
	\midrule
	$C_6$&$\quad [3],[2_2],[2]$&$(s^4t-\frac{1}{2}s^5:s^3t^2:-\frac{3}{2}st^4+t^5)$\\
	$C_7$&$\quad [2_2],[2_2],[2_2]$&$(s^4t-s^5:s^2t^3-\frac{5}{32}s^5:-\frac{125}{128}s^5-\frac{25}{16}s^3t^2-5st^4+t^5)$\\
	\midrule
	$C_8$&$\quad [2_3],[2],[2],[2]$&$(s^4t:s^2t^3-s^5:t^5+2s^3t^2)$\\
	\bottomrule
	\end{tabular*}}
}
	\caption[Plane rational cuspidal curves of degree five.]{Plane rational cuspidal curves of degree five.}
	\label{tab:degree5}
	\end{table}

In two articles from the 1990s Flenner and Zaidenberg construct two series of rational cuspidal curves with three cusps, where one cusp has a relatively high multiplicity \cite{FlZa95, FlZa97}. A third series was found and further exploration of such curves was done by Fenske in \cite{Fen99a,Fent}. The known rational cuspidal curves with three or more cusps can be listed as follows.

\begin{enumerate}[$(I)$]
\item For $d=5$, a rational cuspidal curve with three or more cusps has one of these cuspidal configurations \cite[Theorem 2.3.10, pp.179--182]{Namba}:
\begin{align*}
&[3],[2_2],[2],\\
&[2_2],[2_2],[2_2],\\
&[2_3],[2],[2],[2].
\end{align*}
\item For any $a\geq b \geq 1$ there exists a rational cuspidal plane curve $C$ of degree $d=a+b+2$ with three cusps \cite[Theorem 3.5, p.448]{FlZa95}: 
\begin{equation*}
[d-2],[2_a],[2_b].
\end{equation*}
\item For any $a \geq 1$, there exists a rational cuspidal plane curve $C$ of degree $d=2a+3$ with three cusps \cite[Theorem 1.1, p.94]{FlZa97}: 
\begin{equation*}
[d-3,2_a],[3_a],[2].
\end{equation*}
\item For any $a \geq 1$, there exists a rational cuspidal plane curve $C$ of degree $d=3a+4$ with three cusps \cite[Theorem 1.1, p.512]{Fen99a}: 
\begin{equation*}
[d-4,3_a],[4_a,2_2],[2].
\end{equation*}
\end{enumerate}
All these curves are constructed explicitly by successive Cremona transformations of plane curves of low degree. In cases $(II)$ and $(III)$ it is proved by Flenner and Zaidenberg in \cite{FlZa95,FlZa97} that these are the only tricuspidal curves with maximal multiplicity of this kind. The same is proved by Fenske in \cite{Fen99a} for case $(IV)$ under the assumption that $\chi(\T)=0$. Note that this result is originally proved with $\chi(\T)\leq0$, but by a result by Tono in \cite{Tono05} only $\chi(\T)=0$ is needed. The curves in $(II)$, $(III)$ and $(IV)$ constitute three series of cuspidal curves with three cusps, with infinitely many curves in each series.

The lack of examples of curves with more than four cusps, leads to a conjecture originally proposed by Orevkov.
\begin{conj}
A plane rational cuspidal curve can not have more than four cusps.
\end{conj}

Further research by Fenske in \cite[Section 5]{Fent} and Piontkowski in \cite{Piontkowski} implies that there are no other tricuspidal curves, hence the above conjecture was extended \cite{Piontkowski}.
\begin{conj}
A plane rational cuspidal curve can not have more than four cusps. If it has three or more cusps, then it is given in the above list.
\end{conj}

Associated to these results, in \cite{FlZa94} Flenner and Zaidenberg conjectured that $\Q$-acyclic affine surfaces $Y$ with logarithmic Kodaira dimension $\ol{\kappa}(Y)= 2$ are \emph{rigid} and \emph{unobstructed}. Complements to rational cuspidal curves with three or more cusps on the projective plane are examples of such surfaces. In that case, the conjecture says that for $(V,D)$ the minimal embedded resolution of the plane curve, the Euler characteristic of the logarithmic tangent sheaf vanishes, $$\chi(\T)=\h^2(V,\T)-\h^1(V,\T)+\h^0(V,\T)=0,$$ and in particular, $\h^1(V,\T)=\h^2(V,\T)=0$. Note that it is known by \cite[Theorem 6]{Iit} that $\h^0(V,\T)=0$. This conjecture is referred to as the \emph{rigidity conjecture}. 

In the study of plane rational cuspidal curves, there has, moreover, been found results bounding the multiplicity of the cusp with the highest multiplicity. This gives restrictions on the possible cuspidal configurations on such a curve. The first result on this matter is by Matsuoka and Sakai in \cite{MatsuokaSakai}, where it is shown that $$d<3\hat{m},$$ where with $m_j$ the multiplicity of the cusp $p_j$, $\hat{m}$ is defined to be $\hat{m}:=\max\{m_j\}$. Improved inequalities were found by Orevkov in \cite{Ore}, where for $\alpha=\frac{3 + \sqrt{5}}{2}$, the bound is improved in general to $$d < \alpha (\hat{m}+1)+\frac{1}{\sqrt{5}},$$ for curves with $\lk=-\infty$, this is further improved to $d < \alpha \hat{m}$, and for curves with $\lk=2$, the bound is $$d < \alpha (\hat{m}+1)-\frac{1}{\sqrt{5}}.$$ 

Note also the result by Tono in \cite{Tono05} on the number of cusps on a plane cuspidal curve. For plane \emph{rational} cuspidal curves, the upper bound on the number of cusps is 8. We deal with this question for cuspidal curves on Hirzebruch surfaces in a subsequent article \cite{MOEONC}.

\section{Notation and preliminary results}\label{PR}
In this section we recall what we mean by a curve on a Hirzebruch surface and state preliminary results that give restrictions on the cuspidal configurations of rational cuspidal curves in this case. 

We first recall some basic facts about the Hirzebruch surfaces. Let $\mathbb{F}_e$ denote the Hirzebruch surface of type $e$ for any $e \geq 0$. Recall that $\Fe$ is a projective ruled surface, with $\Fe=\mathbb{P}(\Os \oplus \Os(-e))$ and morphism $\pi: \mathbb{F}_e \longrightarrow \Pot$. Moreover, $p_a(\mathbb{F}_e)=0$ and $p_g(\mathbb{F}_e)=0$ \cite[Corollary V 2.5, p.371]{Hart:1977}. 

For any $e \geq 0$, the surface $\Fe$ can be considered as the toric variety associated to a fan $\Sigma_e \subset \Z^2$, where the rays of the fan $\Sigma_e$ are generated by the vectors
\begin{displaymath}
v_1=\begin{bmatrix}1\\0 \end{bmatrix}, \quad v_2=\begin{bmatrix}0\\1 \end{bmatrix}, \quad v_3=\begin{bmatrix}-1\\e \end{bmatrix}, \quad v_4=\begin{bmatrix}0 \\-1 \end{bmatrix}.
\end{displaymath}
The coordinate ring of $\Fe$ (see \cite{COX}) is denoted by $S_e := \mathbb{C}[x_{0},x_{1},y_{0},y_{1}]$, where the variables can be given a grading by $\mathbb{Z}^{2}$, 
\begin{eqnarray*}
\deg x_0 & = & (1,0),\\
\deg x_1 & = & (1,0),\\
\deg y_0 & = & (0,1),\\
\deg y_1 & = & (-e,1).
\end{eqnarray*}

Let $S_e(a,b)$ denote the $(a,b)$-graded part of $S_e$,
\begin{displaymath}
S_e(a,b):=\mathrm{H}^0(\Fe,\Os_{\Fe}(a,b)) = \bigoplus_{\substack{\alpha_{0}+\alpha_{1}-e\beta_{1} = a \\ \beta_{0}+\beta_{1} = b}} \mathbb{C}x_{0}^{\alpha_{0}}x_{1}^{\alpha_{1}}y_{0}^{\beta_{0}}y_{1}^{\beta_{1}}.
\end{displaymath}

A reduced and irreducible curve $C$ on $\Fe$ is given as the zero set $\V(F)$ of a reduced and irreducible polynomial $F(x_0,x_1,y_0,y_1) \in S_e(a,b)$. In this case, the polynomial $F$ is said to have bidegree $(a,b)$ and the curve $C$ is said to be of type $(a,b)$.

In the language of divisors, let $\LH$ be a \emph{fiber} of $\pi: \mathbb{F}_e \longrightarrow \Pot$ and $\RH_0$ the \emph{special section} of $\pi$. The Picard group of $\Fe$, $\mathrm{Pic}(\mathbb{F}_e)$, is isomorphic to $\Z \oplus \Z$. We choose $\LH$ and $\RH \sim e\LH+\RH_0$ as a generating set of ${\rm{Pic}}(\Fe)$, and then we have \cite[Theorem V 2.17, p.379]{Hart:1977} $$\LH^2=0, \qquad \LH \ldot \RH=1, \qquad \RH^2=e.$$ The canonical divisor $K$ on $\Fe$ can be expressed as \cite[Corollary V 2.11, p.374]{Hart:1977} $$K \sim (e-2)\LH-2\RH \; \text{ and }\; K^2=8.$$ Any irreducible curve $C\neq \LH,\RH_0$ corresponds to a divisor on $\mathbb{F}_e$ given by \cite[Proposition V 2.20, p.382]{Hart:1977} $$C \sim a\LH+b\RH,\quad b>0, \,a\geq 0.$$

Recall that there are \emph{birational transformations} between these surfaces and the projective plane, and these can be given in a quite explicit way (see \cite{MOEPHD}). Using the birational transformations, we are able to construct curves on one surface from a curve on another surface by taking the strict transform.

A first result concerning cuspidal curves on $\mathbb{F}_e$ regards the genus $g$ of the curve. 
\begin{cor}[Genus formula]\label{genusfe}
A cuspidal curve $C$ of type $(a,b)$ with cusps $p_j$, for $j=1,\ldots,s$, and multiplicity sequences $\ol{m}_j=[m_0,m_1, \ldots, m_{t_j-1}]$ on the Hirzebruch surface $\mathbb{F}_e$ has genus $g$, where $$g=\frac{(b-1)(2a-2+be)}{2}-\sum_{j=1}^s \sum_{i=0}^{t_j-1} \frac{m_i(m_i-1)}{2}.$$
\end{cor}

\begin{proof}
Recall that $C \sim a\LH+b\RH$, $K \sim (e-2)\LH-2\RH$, $L^2=0$, $L \ldot M = 1$ and $M^2=e$. By the general genus formula \cite[Example V 3.9.2, p.393]{Hart:1977}, we have $$g=\frac{(a\LH+b\RH) \ldot (a\LH+b\RH+(e-2)\LH-2\RH)}{2}+1-\sum_{j=1}^s \delta_j,$$ where $\delta$ is the \emph{delta invariant}. This gives
\begin{align*}
g&=\frac{b^2e-2be+ab+be-2a+ab-2b}{2}+1-\sum_{j=1}^s \sum_{i=0}^{t_j-1} \frac{m_i(m_i-1)}{2}\\
&=\frac{(b-1)(2a-2+be)}{2}-\sum_{j=1}^s \sum_{i=0}^{t_j-1} \frac{m_i(m_i-1)}{2}.
\end{align*}
\end{proof}

Secondly, the structure of the Hirzebruch surfaces gives restrictions on the multiplicity sequence of a cusp on a curve on such a surface.
\begin{thm} \label{thm:multb}
Let $p$ be a cusp on a reduced and irreducible curve $C$ of type $(a,b)$ with multiplicity sequence $\ol{m}=[m,m_1,\ldots, m_{t-1}]$ on a Hirzebruch surface $\Fe$. Then $m \leq b$.
\end{thm}

\begin{proof}
The coordinates of the point $p$ determine a unique fiber $L$. By intersection theory, $m \leq (L \cdot C)_p$ \cite[Exercise I 5.4, p.36]{Hart:1977}. Moreover, $(L \cdot C)_p \leq L \ldot C$. By intersection theory again, $L \ldot C = b$. Hence, $m \leq (L \cdot C)_p \leq L \ldot C = b$. 
\end{proof}

Further restrictions on the type of points on a curve on $\Fe$ can be found using Hurwitz's theorem \cite[Corollary IV 2.4, p.301]{Hart:1977}. First, the general result in this situation. 
\begin{thm}[Hurwitz's theorem for $\Fe$]\label{RHFe}
Let $C$ be a curve of genus $g$ and type $(a,b)$, where $b > 0$, on a Hirzebruch surface $\Fe$ with $e>0$. Let $\tilde{C}$ denote the normalization of $C$, and let $\nu$ be the composition of the normalization map $\tilde{C} \rightarrow C$ and the projection map $C \rightarrow \Pot$ of degree $b$. Let $e_p$ denote the ramification index of a point $p \in \tilde{C}$ with respect to $\nu$. Then the following equality holds, $$2b+2g-2 = \sum_{p \in \tilde{C}} (e_p-1).$$ When $e=0$, for curves $C$ of genus $g$ and type $(a,b)$, with $a,b>0$, we have $$2\min\{a,b\}+2g-2 = \sum_{p \in \tilde{C}} (e_p-1).$$
\end{thm}

\begin{proof}
The result follows from \cite[Corollary IV 2.4, p.301]{Hart:1977}. With $\nu$ as above, we get $$2b+2g-2 = \sum_{p \in \tilde{C}} (e_p-1).$$ When $e=0$, use the projection map of lower degree, i.e., $\min\{a,b\}$.
\end{proof}

An immediate corollary gives restrictions on the multiplicities of cusps on a curve.
\begin{cor}\label{RHM}
Let $C$ be a cuspidal curve of type $(a,b)$ and genus $g$ with $s>0$ cusps $p_j$ with multiplicities $m_j$ on a Hirzebruch surface $\Fe$ with $e>0$. Then the following inequality holds, $$2b+2g-2 \geq \sum_{j=1}^s (m_j-1).$$  When $e=0$, the same is true with $\min\{a,b\}$ instead of $b$.
\end{cor}

\begin{proof}
The cusps $p_j$ of $C$ gives branching points with ramification index bigger than or equal to the multiplicity $m_j$, so the result follows from Theorem \ref{RHFe}. 
\end{proof}

\section{Rational cuspidal curves with four cusps}\label{4C}
In this section we give examples of rational cuspidal curves with four cusps on Hirzebruch surfaces, and our aim is to shed light on the question of how many and what kind of cusps a rational cuspidal curve on a Hirzebruch surface can have. 


We are not able to construct many rational cuspidal curves with four cusps on the Hirzebruch surfaces. Indeed, on each $\Fe$, with $e \geq 0$, we construct one infinite series of rational cuspidal curves with four cusps. On $\PP$ and $\Fen$ we construct another three infinite series of rational cuspidal curves with four cusps, and on $\mathbb{F}_2$ we construct a single additional rational cuspidal curve with four cusps.

The following theorem presents the series of rational fourcuspidal curves that consists of curves on all the Hirzebruch surfaces. In Appendix \ref{Appendix:A} we construct some of the curves from a plane fourcuspidal curve using birational maps and \begin{verbatim}Maple\end{verbatim} \cite{Maple}. 
\begin{thm} \label{thm:4cusp}
For all $e \geq 0$ and $k \geq 0$, except for the pair ${(e,k)=(0,0)}$, there exists on the Hirzebruch surface $\Fe$ a rational cuspidal curve $C_{e,k}$ of type $(2k+1,4)$ with four cusps and cuspidal configuration 
$$[4_{k-1+e},2_3],[2],[2],[2].$$
\end{thm}

\begin{proof}
We will show that for each $e \geq 0$ there is an infinite series of curves on $\Fe$, and we show this by induction on $k$. The proof is split in two, and we treat the case of $k$ odd and even separately. We construct the series of curves $C_{e,0}$ for $e \geq 1$, and then we construct the initial series $C_{e,1}$ and $C_{e,2}$, with $e \geq0$. We only treat the induction to prove the existence of $C_{e,k}$ for odd values of $k$, as the proof for even values of $k$ is completely parallel. 

Let $C$ be the rational cuspidal curve of degree 5 on $\Po$ with cuspidal configuration $[2_3],[2],[2],[2]$. Let $p$ be the cusp with multiplicity sequence $[2_3]$, and let $T$ be the tangent line to $C$ at $p$. Then $T \cdot C=4 p+r$, with $r$ a smooth point on $C$. Blowing up at $r$, the strict transform of $C$ is a curve $C_{1,0}$ of type $(1,4)$ on $\Fen$ with cuspidal configuration $[2_3],[2],[2],[2]$. Letting $T_{1,0}$ denote the strict transform of $T$ and $p_{1,0}$ the strict transform of $p$, we have ${T_{1,0} \cdot C_{1,0}=4 p_{1,0}}$. We observe that $p_{1,0}$ is fiber tangential. Let $E_{1}$ denote the special section on $\Fen$, and let $s_{0,1}=E_{1} \cap T_{1,0}$. 

From $C_{1,0}$ we can proceed with the construction of curves on Hirzebruch surfaces in three ways. 

First, we show by induction on $e$ that the curves $C_{e,0}$ exist on the Hirzebruch surfaces $\Fe$, for all $e \geq 1$. We have already seen that $C_{1,0}$ exists on $\Fen$, and that there exists a fiber $T_{1,0}$ with the property that ${T_{1,0} \cdot C_{1,0}=4 p_{1,0}}$ for the first cusp $p_{1,0}$. Now assume $e \geq 2$ and that the curve $C_{e-1,0}$ of type $(1,4)$ exists on $\mathbb{F}_{e-1}$ with cuspidal configuration $[4_{e-2},2_3],[2],[2],[2]$, where $p_{e-1,0}$ denotes the first cusp and $T_{e-1,0}$ has the property that ${T_{e-1,0} \cdot C_{e-1,0}=4 p_{e-1,0}}$. Then, with $E_{e-1}$ the special section of $\mathbb{F}_{e-1}$, blowing up at the intersection $s_{e-1,0} \in E_{e-1} \cap T_{e-1,0}$ and contracting $T_{e-1,0}$, we get $C_{e,0}$ on $\Fe$ of type $(1,4)$ with cuspidal configuration $[4_{e-1},2_3],[2],[2],[2]$. Moreover, we note that there exists a fiber $T_{e,0}$ with ${T_{e,0} \cdot C_{e,0}=4 p_{e,0}}$. So the series exists on all $e \geq 1$ for $k=0$.

Second, note that from the curve $C_{1,0}$ on $\Fen$ it is possible to construct the curve $C_{0,1}$ on $\PP$ by blowing up at $p_{1,0}$ before contracting $T_{1,0}$. The curve $C_{0,1}$ is a curve of type $(3,4)$ with cuspidal configuration ${[2_3],[2],[2],[2]}$, and there is a fiber $T_{0,1}$ such that ${T_{0,1}\ldot C_{0,1}=4p_{0,1}}$. Blowing up at a point ${s_{0,1} \in T_{0,1} \setminus \{p_{0,1}\}}$ and contracting $T_{0,1}$ result in the curve $C_{1,1}$ on $\Fen$ of type $(3,4)$ with cuspidal configuration $[4,2_3],[2],[2],[2]$. Moreover, there exists a fiber $T_{1,1}$ with ${T_{1,1} \cdot C_{1,1}=4 p_{1,1}}$ and $p_{1,1} \notin E_1$. The same induction on $e$ as above proves that the series exists for $k=1$.

Third, note that from the curve $C_{1,0}$ on $\Fen$ it is possible to construct the curve $C_{0,2}$ on $\PP$ by blowing up at a point ${t_{1,0} \in T_{1,0} \setminus \{p_{1,0}, s_{1,0}\}}$ before contracting $T_{1,0}$. The curve $C_{0,2}$ is a curve of type $(5,4)$ with cuspidal configuration $[4,2_3],[2],[2],[2]$, and there is a fiber $T_{0,2}$ such that ${T_{0,2}\cdot C_{0,2}=4p_{0,2}}$. Blowing up at a point ${s_{0,2} \in T_{0,2} \setminus \{p_{0,2}\}}$ and contracting $T_{0,2}$ give the curve $C_{1,2}$ on $\Fen$ of type $(5,4)$ with cuspidal configuration ${[4_2,2_3],[2],[2],[2]}$. Moreover, there exists a fiber $T_{1,2}$ with ${T_{1,2} \cdot C_{1,2}=4 p_{1,2}}$ and $p_{1,2} \notin E_1$. The same induction on $e$ as above proves that the series exists for $k=2$.

Next assume $k \geq 3$, with $k$ odd, and that there exists a series of curves $C_{e,k-2}$ of type $(2k-3,4)$ on $\Fe$ for all $e \geq 0$ with cuspidal configuration $[4_{e+k-3},2_3],[2],[2],[2]$. Then, in particular, the curve $C_{1,k-2}$ on $\Fen$ with cuspidal configuration \linebreak $[4_{k-2},2_3],[2],[2],[2]$ exists. Moreover, there exists a fiber $T_{1,k-2}$ on $\Fen$ such that ${T_{1,k-2} \cdot C_{1,k-2}=4 p_{1,k-2}}$, where $p_{1,k-2}$ denotes the cusp with multiplicity sequence $[4_{k-2},2_3]$. With $E_1$ the special section on $\Fen$, let ${s_{1,k-2} \in E_1 \cap T_{1,k-2}}$. We now blow up at a point $t_{1,k-2} \in T_{1,k-2} \setminus \{p_{1,k-2},s_{1,k-2}\}$ and subsequently contract $T_{1,k-2}$. This gives the curve $C_{0,k}$ on $\PP$ of type $(2k+1,4)$ with cuspidal configuration $[4_{k-1},2_3],[2],[2],[2]$. With $T_{0,k}$ the strict transform of the exceptional line of the latter blowing up, we have ${T_{0,k} \cdot C_{0,k}=4 p_{0,k}}$. Blowing up at a point $s_{0,k} \in T_{0,k} \setminus \{p_{0,k}\}$ and contracting $T_{0,k}$ gives the curve $C_{1,k}$ on $\Fen$ of type $(2k+1,4)$ with cuspidal configuration $[4_k,2_3],[2],[2],[2]$. Moreover, there is a fiber $T_{1,k}$ with the property that ${T_{1,k}\cdot C_{1,k}=4p_{1,k}}$. With the same induction on $e$ as above, we get the series of curves $C_{e,k}$.
\end{proof}

There are three infinite series of rational fourcuspidal curves that can be found on the Hirzebruch surfaces $\PP$ and $\Fen$. 

Before we list these three series, we consider the rational cuspidal curves with four cusps on $\Fen$ that we can get by blowing up a single point on $\Po$. These curves represent examples from the series.
\begin{thm}
Let $C$ be the rational cuspidal curve with four cusps of degree 5 on $\Po$. The following rational cuspidal curves with four cusps on $\mathbb{F}_1$ can be constructed from $C$ by blowing up a single point on $\Po$. 
\begin{table}[H]
  \renewcommand\thesubtable{}
  \setlength{\extrarowheight}{2pt}
\centering
	{\noindent \begin{tabular*}{0.75\linewidth}{@{}c@{\extracolsep{\fill}}c@{}c@{}l}
	\tr
	 \textnormal{\textbf{\# Cusps}}  & \textnormal{\textbf{Curve}} & \textnormal{\textbf{Type}} & \textnormal{\textbf{Cuspidal configuration}} \\
	\tr
	\multirow{3}{0pt}{4} &  $C_1$ & $(0,5)$& $[2_3],[2],[2],[2]$  \\
	 &$C_2$ &$(1,4)$ & $[2_3],[2],[2],[2]$ \\
	 &$C_3$ &$(2,3)$  & $[2_2],[2],[2],[2]$ \\
	\br
	\end{tabular*}}
	\caption {Rational cuspidal curves on $\mathbb{F}_1$ with four cusps.}
	\label{tab:prccw4}
	\end{table}
\end{thm}

\begin{proof}
The curve $C_1$ is constructed by blowing up any point $r$ on $\Po \setminus C$. Note that if $r$ is on the tangent line to a cusp on $C$, then $C_1$ has cusps that are fiber tangential. If $r$ is only on tangent lines of smooth points on $C$, then $C_1$ has smooth fiber tangential points. 

The curve $C_2$ is constructed by blowing up any smooth point $r$ on $C$. Again, if $r$ is on a tangent line of $C$, $C_2$ will have points that are fiber tangential.

The curve $C_3$ is constructed by blowing up the cusp with multiplicity sequence $[2_3]$.
\end{proof}

\noi The fact that we can construct curves where the curve have nontransversal intersection with some fiber(s) is crucial in the later constructions. Although we do not get new cuspidal configurations in this first step, cusps can sometimes be constructed later on.

We now give the three series of rational cuspidal curves with four cusps on $\PP$ and $\Fen$. For notational purposes we denote these surfaces by $\Fh$, with $h \in \{0,1\}$, in the theorems.
\begin{thm}
For $h\in\{0,1\}$ and all integers $k\geq 0$, except the pair $(h,k)=(0,0)$, there exists on the Hirzebruch surface $\Fh$ a rational cuspidal curve $C_{h,k}$ of type $(3k+1-h,5)$ with four cusps and cuspidal configuration 
$$[4_{2k-1+h},2_3],[2],[2],[2].$$
\end{thm}

\begin{proof}
The proof is by construction and induction on $k$. Let $C$ be a rational cuspidal curve of degree 5 on $\Po$ with cuspidal configuration ${[2_3],[2],[2],[2]}$. Let $p$ be the cusp with multiplicity sequence $[2_3]$, and let $T$ be the tangent line to $C$ at $p$. There is a smooth point $r \in C$, such that $T \cdot C=4 p+r$. Blowing up at any point $t \in T \setminus \{p,r\}$, we get the curve $C_{1,0}$ of type $(0,5)$ and cuspidal configuration ${[2_3],[2],[2],[2]}$ on $\Fen$. Moreover, with $T_{1,0}$ the strict transform of $T$ and $p_{1,0}$, $r_{1,0}$ the strict transforms of the points $p$ and $r$, we have ${T_{1,0} \cdot C_{1,0}=4 p_{1,0}+r_{1,0}}$.

Now assume that the curve $C_{1,k-1}$ of type $(3(k-1),5)$ exists on $\Fen$ with cuspidal configuration $[4_{2(k-1)},2_3],[2],[2],[2]$, and the intersection $T_{1,k-1} \cdot C_{1,k-1}=4 p_{1,k-1}+r_{1,k-1}$ for a fiber $T_{1,k-1}$ and points as above. Then blowing up at $r_{1,k-1}$ and contracting $T_{1,k-1}$, we get a curve $C_{0,k}$ on $\PP$ of type $(3k+1,5)$ and cuspidal configuration $[4_{2k-1},2_3],[2],[2],[2]$. Moreover, there is a fiber $T_{0,k}$ with the property that $T_{0,k} \cdot C_{0,k}=4 p_{0,k}+r_{0,k}$. Blowing up at $r_{0,k}$ and contracting $T_{0,k}$, we get a rational cuspidal curve $C_{1,k}$ of type $(3k,5)$ on $\Fen$ with cuspidal configuration $[4_{2k},2_3],[2],[2],[2]$.
\end{proof}

\begin{thm}
For $h\in\{0,1\}$ and all integers $k\geq 0$, except the pair $(h,k)=(0,0)$, there exists on the Hirzebruch surface $\Fh$ a rational cuspidal curve of type $(2k+2-h,4)$ with four cusps and cuspidal configuration 
$$[3_{2k-1+h},2],[2_3],[2],[2].$$
\end{thm}

\begin{proof}
The proof is by construction and induction on $k$. Let $C$ be the rational cuspidal curve of degree 5 on $\Po$ with cuspidal configuration ${[2_3],[2],[2],[2]}$. Let $q$ be one of the cusps with multiplicity sequence $[2]$, and let $T$ be the tangent line to $C$ at $q$. Then there are smooth points $r,s \in C$, such that ${T \cdot C=3 q+r+s}$. Blowing up at $s$, we get the curve $C_{1,0}$ of type $(1,4)$ and cuspidal configuration $[2_3],[2],[2],[2]$ on $\Fen$. Moreover, with $T_{1,0}$ the strict transform of $T$ and $p_{1,0}$, $r_{1,0}$ the strict transforms of the points $p$ and $r$, we have $T_{1,0} \cdot C_{1,0}=3 p_{1,0}+r_{1,0}$.

Now assume that the curve $C_{1,k-1}$ of type $(2k-1,4)$ exists on $\Fen$ with cuspidal configuration $[3_{2k-2},2],[2_3],[2],[2]$, and the intersection $T_{1,k-1} \cdot C_{1,k-1}=3 p_{1,k-1}+r_{1,k-1}$ for a fiber $T_{1,k-1}$ and points as above. Then blowing up at $r_{1,k-1}$ and contracting $T_{1,k-1}$, we get a curve $C_{0,k}$ on $\PP$ of type $(2k+2,4)$ and cuspidal configuration $[3_{2k-1},2],[2_3],[2],[2]$. Moreover, there is a fiber $T_{0,k}$ with the property that $T_{0,k} \cdot C_{0,k}=3 p_{0,k}+r_{0,k}$. Blowing up at $r_{0,k}$ and contracting $T_{0,k}$, we get a rational cuspidal curve $C_{1,k}$ of type $(2k+1,4)$ on $\Fen$ with cuspidal configuration $[3_{2k},2],[2_3],[2],[2]$.
\end{proof}




\begin{thm}\label{2erne}
For $h\in\{0,1\}$, all integers $k \geq 2$, and every choice of $n_j \in \mathbb{N}$, with $j=1,\ldots,4$, such that $\sum_{j=1}^4 n_j=2k+h$, there exists on the Hirzebruch surface $\Fh$ a rational cuspidal curve of type $(k+1-h,3)$ with four cusps and cuspidal configuration $$[2_{n_1}],[2_{n_2}],[2_{n_3}],[2_{n_4}].$$
\end{thm}

\begin{proof}
We prove the existence of the curves on $\PP$ by induction on $k$. In the proof we show that any curve on $\Fen$ can be reached from a curve on $\PP$ by an elementary transformation, hence we prove the theorem for $h\in \{0,1\}$. 

First we observe that a choice of $n_j$ such that the condition $\sum_{j=1}^4 n_j=2k$ means that either all four $n_j$ are odd, two are odd and two are even, or all four are even. We split the proof into these three cases, and prove only the first case completely. The other two can be dealt with in the same way once we have proved the existence of a first curve.

We now prove the theorem when all $n_j$ are odd. Let $C$ be a rational cuspidal curve on $\Po$ of degree 4 with three cusps and cuspidal configuration ${[2],[2],[2]}$ for cusps $p_j$, $j=1,2,3$. Let $p_4$ be a general smooth point on $C$ and let $T$ be the tangent line to $C$ at $p_4$. Then ${T \cdot C = 2 p_4+t_1+t_2}$, where $t_1,t_2$ are two smooth points on $C$. Blowing up at $t_1$ and $t_2$ and contracting $T$, we get a rational cuspidal curve on $\PP$ of type $(3,3)$ with four ordinary cusps.

Fixing notation, we say that we have a curve $C_2$ on $\PP$ of type $(3,3)$ and four cusps $p^2_j$, $j=1,\ldots,4$, all with multiplicity sequence $[2]$. Since the choice of $p_4 \in \Po$ was general, there are four $(1,0)$-curves $L^2_j$ such that $$L^2_j \cdot C_2=2 p^2_j+r^2_j,$$ for smooth points $r^2_j \in C_2$. Now assume that we have a curve $C_{k-1}$ on $\PP$ of type $((k-1)+1,3)$, with cuspidal configuration ${[2_{n_1-2}],[2_{n_2}],[2_{n_3}],[2_{n_4}]}$ such that all $n_j$ are odd, and such that there exist fibers $L^{k-1}_j$ with $$L^{k-1}_j \cdot C_{k-1}=2 p^{k-1}_j+r^{k-1}_j,$$ for smooth points $r^{k-1}_j$ on $C_{k-1}$. 

We blow up at $r^{k-1}_1$, contract the corresponding $L^{k-1}_1$ and get a curve $C_{1,k-1}$ on $\Fen$ of type $(k-1,3)$ with cuspidal configuration ${[2_{n_1-1}],[2_{n_2}],[2_{n_3}],[2_{n_4}]}$. Moreover, since $r^{k-1}_1$ was not fiber tangential, we have that $r^{1,k-1}_1 \notin E_1$, and the strict transform of the exceptional fiber of the blowing up, $L^{1,k-1}_1$, has intersection with $C_{1,k-1}$, $$L^{1,k-1}_1 \cdot C_{1,k-1}=2 p^{1,k-1}_1+r^{1,k-1}_1.$$ Blowing up at $r^{1,k-1}_1$ and contracting $L^{1,k-1}_1$ bring us back to $\PP$ and a curve $C_{k}$ of type $(k+1,3)$ and cuspidal configuration $[2_{n_1}],[2_{n_2}],[2_{n_3}],[2_{n_4}]$. This takes care of the case when all $n_j$ are odd.

To prove the theorem when two $n_j$ are even or all $n_j$ are even, we only show that there exist curves on $\PP$ of the right type and cuspidal configurations $[2_2],[2_2],[2],[2]$ and ${[2_2],[2_2],[2_2],[2_2]}$. The rest of the argument is then similar to the above. To get the first curve, we blow up $C_2$ in $r^2_1$ and $r^2_2$ and contract $L^2_1$ and $L^2_2$. This is a curve $C_3$ of type $(4,3)$ with cuspidal configuration ${[2_2],[2_2],[2],[2]}$. The curve is on $\PP$ since it can be shown with direct calculations in \verb+Maple+ that $r^2_1$ and $r^2_2$ are not on the same $(0,1)$-curve. To get the second curve, we blow up at the analogous $r^3_3$ and $r^3_4$ on the curve $C_3$, before contracting $L^3_3$ and $L^3_4$. We are again on $\PP$ by a similar argument to the above, and the curve $C_4$ is of type $(5,3)$ and has cuspidal configuration $[2_2],[2_2],[2_2],[2_2]$.
\end{proof}

\noi Note that the construction of the curves in Theorem \ref{2erne} can also be done from the rational cuspidal cubic on $\Po$.
\begin{proof}[Alternative proof of Theorem \ref{2erne}]
Let $C$ be the rational cuspidal cubic on $\Po$. Let $s$ be a general point on $\Po$, where general here means that $s$ is neither on $C$, nor the tangent line to the cusp, nor the tangent line to the inflection point on $C$. For example we can choose ${y^2z-x^3}$ as the defining polynomial of $C$, and take ${s=(0:1:1)}$. Then the polar curve of $C$ with respect to the point $s$, given by the defining polynomial ${2yz+y^2}$, intersects $C$ in three smooth points, ${(2^{\frac{2}{3}}:-2:1)}$, ${(2^{-\frac{1}{3}}(-1+\sqrt{3}i):-2:1)}$ and ${(2^{-\frac{1}{3}}(-1-\sqrt{3}i):-2:1)}$. Blowing up at $s$ brings us to $\Fen$ and a curve of type $(0,3)$ with one ordinary cusp, say $p_4$. We additionally have three fibers $L_j$, $j=1,\ldots,3$, with the property that ${L_j \cdot C=2 p_j+r_j}$ for smooth points $p_j$ and $r_j$ on $C$. Blowing up at the $r_j$'s and contracting the $L_j$'s, we get the desired series of curves. 
\end{proof}

The series in Theorem \ref{2erne} can be extended to a series of rational cuspidal curves with less than four cusps in an obvious way. We state this as a corollary.
\begin{cor}
For $h\in\{0,1\}$, all integers $k \geq 0$, and every choice of $n_j \in \mathbb{N} \cup \{0\}$, with ${j=1,\ldots,4}$, such that ${\sum_{j=1}^4 n_j=2k+h}$, there exists on the Hirzebruch surface $\Fh$ a rational cuspidal curve of type $(k+1-h,3)$ with $s \in \{0,1,2,3,4\}$ cusps and cuspidal configuration $$[2_{n_1}],[2_{n_2}],[2_{n_3}],[2_{n_4}].$$
\end{cor}

\begin{proof}
These curves can be constructed from the curves in Theorem \ref{2erne} by a similar construction. In order to construct the curves with less than four cusps we have to blow up cusps on the curves in the series from Theorem \ref{2erne}.
\end{proof}

Last in this section we provide an example of a curve not represented in any of the above series. This is the only example we have found of such a curve, and in particular the only such curve on $\mathbb{F}_2$.
\begin{thm}
On $\mathbb{F}_2$ there exists a rational cuspidal curve of type $(0,3)$ with four cusps and cuspidal configuration 
$$[2],[2],[2],[2].$$
\end{thm}

\begin{proof}
Let $C$ be the plane rational fourcuspidal curve of degree $5$. Let $p$ be the cusp $[2_3]$ and $p_i$, $i=1,2,3$, the cusps with multiplicity sequence $[2]$. Let $T$ be the tangent line to $C$ at $p$. Let $L_i$ denote the line through $p$ and $p_i$, with $i=1,2,3$. There are smooth points $s$ and $r_i$, $i=1,2,3$, on $C$, such that 
$$T \cdot C=4 p + s, \qquad L_i \cdot C=2 p+2 p_i + r_i.$$

Blowing up at $p$ gives a $(2,3)$-curve on $\Fen$ with cuspidal configuration ${[2_2],[2],[2],[2]}$. Let $C'$ denote the strict transform of $C$, $T'$ and $L_i'$ the strict transforms of $T$ and $L_i$, and let $E'$ be the special section on $\Fen$. Let $p'$  be the cusp $[2_2]$, $p_i'$ the other cusps, and $s'$ and $r_i'$ the strict transforms of the points $s$ and $r_i$. Then we have the following intersections, 
$$E' \cdot C'=2 p', \qquad T'\cdot C'=2 p'+s', \qquad L_i'\cdot C'=2 p_i' + r_i'.$$ Since $p' \in E'$, blowing up at $p'$ and contracting $T'$, we get a cuspidal curve on $\mathbb{F}_2$ of type $(0,3)$ and cuspidal configuration ${[2],[2],[2],[2]}$.

\end{proof}

\section{Associated results}\label{AR}
The main result in this section is that the rigidity conjecture proposed by Flenner and Zaidenberg for plane rational cuspidal curves can \emph{not} be extended to the case of rational cuspidal curves on Hirzebruch surfaces. First in this section we state and prove two lemmas for rational cuspidal curves on Hirzebruch surfaces, the first analogous to a lemma by Flenner and Zaidenberg \cite[Lemma 1.3, p.148]{FlZa94}, and the other a lemma bounding the sum of the so-called $M$-numbers of the curve. Second, we use these lemmas to give an explicit formula for $\chi(\T)$ in this case. We calculate this value for the curves constructed in Section \ref{4C}, and with that we provide examples of curves for which $\chi(\T) \neq 0$. Third, we use the two mentioned lemmas and other results to find a lower bound on the highest multiplicity of a cusp on a rational cuspidal curve on a Hirzebruch surface. Last in this section, we investigate real cuspidal curves on Hirzebruch surfaces.

\subsection{Two lemmas}
We now state and prove two lemmas for rational cuspidal curves on Hirzebruch surfaces. First, we prove a lemma that is a variant of \cite[Lemma 1.3, p.148]{FlZa94}.

\begin{lem}\label{lemflzadef}\label{lemflzadeffe}
Let $C$ be a rational cuspidal curve on $\Fe$. Let $(V,D)$ be the minimal embedded resolution of $C$, and let $K_V$ be the canonical divisor on $V$. Moreover, let $D_1, \ldots, D_r$ be the irreducible components of $D$, $\Theta_V$ the tangent sheaf of $V$, $\mathscr{N}_{D/V}$ the normal sheaf of $D$ in $V$, and let $c_2$ be the second Chern class of $V$. Then the following hold.
\begin{enumerate}
\item[$(0)$]
$D$ is a rational tree.
\item[$(1)$]
$\chi(\Theta_V)=8-2r$.
\item[$(2)$]
$K_V^2=9-r$.
\item[$(3)$]
$c_2:=c_2(V)=3+r$.
\item[$(4)$]
$\displaystyle \chi \left(\bigoplus \mathscr{N}_{D_i/V}\right)=r+\sum_{i=1}^rD_i^2$.
\item[$(5)$]
$\displaystyle \chi(\T)=K_V \ldot (K_V+D)-1.$
\end{enumerate}
\end{lem} 

\begin{proof} Note that the proof is very similar to the proof of \cite[Lemma 1.3, p.148]{FlZa94}, only small details are changed. 

\begin{enumerate}
\item[$(0)$]
Since $(V,D)$ is the minimal embedded resolution of $C$, $D$ is an SNC-divisor. Since $C$ is a rational curve, $\tilde{C}$ is smooth, and all exceptional divisors are smooth and rational, then all components of $D$ are smooth and rational. The dual graph of $D$, say $\Gamma$, is necessarily a connected graph, and since $C$ is cuspidal, $\Gamma$ will not contain cycles. Thus, $D$ is by definition a rational tree.

\item[$(3)$]
Since $V$ is obtained by $r-1$ blowing ups, we have that the Chern class $c_2:=c_2(V)$ is $$c_2(V)=c_2(\Fe)+r-1.$$ Moreover, by \cite[Corollary V 2.5, p.371]{Hart:1977} we have that $\chi(\mathscr{O}_{\Fe})=1$. With $K$ the canonical divisor on $\Fe$, we apply the formula in \cite[Remark V 1.6.1, p.363]{Hart:1977}, 
\begin{align*}
12\chi(\mathscr{O}_{\Fe})&=K^2+c_2(\Fe)\\
12 &=8 + c_2(\Fe).\\
\end{align*}
We get $c_2(\Fe)=4$, hence,
\begin{align*}
c_2(V)&=4+r-1\\
&=3+r.
\end{align*}

\item[$(2)$]
We have by \cite[Proposition V 3.4, p.387]{Hart:1977} that $\chi(\mathscr{O}_V)=\chi(\mathscr{O}_{\Fe})=1$. By the formula in \cite[Remark V 1.6.1, p.363]{Hart:1977} again, we get 
\begin{align*}
K_V^2&=12\chi(\mathscr{O}_V)-c_2\\
&=12-(3+r)\\
&=9-r.
\end{align*}

\item[$(4)$]
Since $D_i$ is a rational curve on the surface $V$ for all $i$, we have that $g(D_i)=0$. By \cite[Proof of Proposition II 8.20, p.182]{Hart:1977} we have that $$\mathscr{N}_{D_i/V}\cong \mathscr{L}(D_i) \otimes \Os_{D_i}.$$
Hence, by the Riemann--Roch theorem for curves \cite[p.362]{Hart:1977}, 
\begin{align*}
\displaystyle \chi \left(\bigoplus \mathscr{N}_{D_i/V}\right)&= \chi \left(\bigoplus \mathscr{L}(D_i) \otimes \Os_{D_i}\right)\\
&=\sum_{i=1}^r\left(D_i^2+1\right)\\
&=r+\sum_{i=1}^rD_i^2.
\end{align*}

\item[$(1)$]
By the Hirzebruch--Riemann--Roch theorem for surfaces \cite[Theorem A 4.1, p.432]{Hart:1977}, we have that for any locally free sheaf $\mathscr{E}$ on $V$ of rank $s$ with Chern classes $c_i(\mathscr{E})$,
$$\chi(\mathscr{E})=\frac{1}{2}c_1(\mathscr{E}) \ldot \bigl(c_1(\mathscr{E})+c_1(\Theta_V)\bigr)-c_2(\mathscr{E})+s\cdot \chi\left(\Os_V\right).$$

Moreover, by \cite[Example A 4.1.2, p.433]{Hart:1977}, $\Theta_V$ has rank $s=2$ and ${c_1(\Theta_V)=-K_V}$. We have by the previous results,
\begin{align*}
\chi(\Theta_V)&=\frac{1}{2}(-K_V) \ldot (-2K_V)-c_2(\Theta_V)+2\chi\left(\Os_V\right)\\
&=K_V^2-c_2+2\\
&=9-r-(3+r)+2\\
&=8-2r.
\end{align*}

\item[$(5)$]
Observe first that since $D$ is an SNC-divisor, we have by direct calculation
\begin{align*}\displaystyle D^2&=\sum_{i=1}^rD_i^2 + \sum_{i\neq j}D_iD_j\\
&=\sum_{i=1}^rD_i^2+(1+2(r-2)+1)\\
&=\sum_{i=1}^rD_i^2+2r-2.
\end{align*}

Since $D$ is an effective divisor, we have by definition that ${p_a(D)=1-\chi(\Os_D)}$. Since $D$ is a rational tree, by \cite[Lemma 1.2, p.148]{FlZa94}, $p_a(D)=0$. So by the adjunction formula \cite[Exericise V 1.3, p.366]{Hart:1977}, $$K_V \ldot D=-D^2-2.$$

Using the additivity of $\chi$ on the short exact sequence (see \cite[pp.147,162]{FlZa94}),
$$0 \longrightarrow \T \longrightarrow \Theta_V \longrightarrow \bigoplus \mathscr{N}_{D_i/V} \longrightarrow 0,$$ 
and the above results and remarks, we get

\begin{align*}
\chi(\T)&=\chi(\Theta_V)-\chi \left(\bigoplus \mathscr{N}_{D_i/V}\right)\\
&=(8-2r)-\Bigl(r+\sum_{i=1}^rD_i^2 \Bigr)\\
&=8-2r-(r+D^2-2r+2)\\
&=6-r-D^2\\
&=K_V^2-D^2-3\\
&=K_V^2+2K_V\ldot D-2(-D^2-2)-D^2-3\\
&=(K_V+D)^2+1\\
&=K_V \ldot (K_V+D)-1.
\end{align*}

\end{enumerate}
\end{proof}

The second lemma bounds the sum of the so-called $M$-numbers by the type of the curve, and this work is inspired by Orevkov (see \cite{Ore}). 

For a cusp $p$, the associated \emph{$M$-number} can be defined as
\begin{equation*}
{M} := \eta+\omega-1,
\end{equation*}
where
\begin{equation*}
\eta := \sum_{i=0}^{t-1}\left(m_i-1\right),
\end{equation*}
and $\omega$ is the number of blowing ups in the minimal embedded resolution which center is an intersection point of the strict transforms of two exceptional curves of the resolution, i.e., an \emph{inner} blowing up.

Moreover, the $M$-number can be expressed in terms of the multiplicity sequence $\ol{m}$, 
\begin{equation*}
{M} = \sum_{i=0}^{t-1}\left(m_i-1\right)+\sum_{i=1}^{t-1}\left(\left\lceil\frac{m_{i-1}}{m_i}\right\rceil-1\right)-1,
\end{equation*} 
where $\lceil a \rceil$ is the smallest integer $\geq a$. See \cite[Definition 1.5.23, p.44]{Fent} and \cite[p.659]{Ore} for more details.

Before we state and prove the lemma, recall that $\lkkf$ denotes the logarithmic Kodaira dimension of the complement to $C$ in $\Fe$ (see \cite{Iitaka}).

\begin{lem}\label{kod0M}
For a rational cuspidal curve $C$ of type $(a,b)$ on $\Fe$ with $s$ cusps and $\lkkf\geq 0$, we have
$$\sum_{j=1}^s M_j \leq 2(a+b)+be.$$
\end{lem}

\begin{proof}
Let $(V,D)$ and ${\sigma=\sigma_1 \circ \ldots \circ \sigma_t}$ be the minimal embedded resolution of $C$. Write ${\sigma^*(C)=\tilde{C}+\sum_{i=1}^t m_{i-1}E_i}$, with $\tilde{C}$ the strict transform of $C$ under $\sigma$, $m_{i-1}$ the multiplicity of the center of $\sigma_i$ and $E_i$ the exceptional curve of $\sigma_i$. Then by induction and \cite[Proposition V 3.2, p.387]{Hart:1977} we find that
\begin{align*}
\tilde{C}^2 &=\Bigl(\sigma^*(C)-\sum_{i=1}^t m_{i-1}E_i\Bigr)^2 \notag\\
&=C^2-\sum_{i=0}^{t-1}m_i^2 \notag \\
&=b^2e+2ab-\sum_{i=0}^{t-1}m_i^2.
\end{align*} 

\noi By the genus formula, we may rewrite this,
\begin{align*}
\tilde{C}^2&=b^2e+2ab-\sum_{i=0}^{t-1}m_i^2\\
&=be+2(a+b)-2-\sum_{i=0}^{t-1} m_i.
\end{align*}

\noi Moreover, we have that for ${D=\tilde{C}+\sum_{i=1}^t E_i'}$, with $E_i'$ the strict transform of $E_i$ under the composition $\sigma_{i+1} \circ \cdots \circ \sigma_t$,
\begin{align*}
D^2&=\tilde{C}^2+2\tilde{C} \ldot \Bigl(\sum_{i=1}^t E_i'\Bigr)+\Bigl(\sum_{i=1}^t E_i'\Bigr)^2 \notag\\
&=\tilde{C}^2+2s+\Bigl(\sum_{i=1}^t E_i'\Bigr)^2.
\end{align*}

\noi Now we split the latter term in this sum into the sum of the strict transforms of the exceptional divisors for each cusp, $$\sum_{i=1}^t E_i'=\sum_{j=1}^s E_{p_j},$$ where $s$ denotes the number of cusps.
By \cite[Lemma 2, p.235]{MatsuokaSakai}, we have 
$$\omega_j=-E_{p_j}^2-1.$$ 

\noi Combining the above results, we get
\begin{align*}
D^2&=be+2(a+b)-2- \sum_{i=0}^{t-1}m_i+2s-\sum_{j=1}^s (\omega_j+1)\\
&=be+2(a+b)-2-\sum_{i=0}^{t-1} m_i-\sum_{j=1}^s (\omega_j-1).
\end{align*}

\noi By the proof of Lemma \ref{lemflzadef}, we have $$6-r-D^2=(K_V+D)^2+1.$$ Note that $r$ denotes the number of components of the divisor $D$. This number is equal to the total number of blowing ups needed to resolve the singularities, plus one component from the strict transform of the curve itself. Following the notation established, we have $r=t+1$. Moreover, by assumption, $\lkkf \geq 0$. By the logarithmic Bogomolov--Miyaoka--Yau-inequality (B--M--Y-inequality) in the form given by Orevkov \cite[Theorem 2.1, p.660]{Ore} and the topological Euler characteristic of the complement to the curve (see \cite{MOEONC}), we then have that $$(K_V+D)^2 \leq 6.$$

\noi So we get
\begin{align*}
0&\leq 1+r+D^2\\
& \leq 1+r+be+2(a+b)-2-\sum_{i=0}^{t-1} m_i-\sum_{j=1}^s (\omega_j-1)\\
& \leq -1+1+be+2(a+b)-\sum_{i=0}^{t-1}(m_i-1)-\sum_{j=1}^s (\omega_j-1)\\
&\leq be+2(a+b)-\sum_{j=1}^s M_j.
\end{align*}

\noi Hence,
$$\sum_{j=1}^s M_j \leq 2(a+b)+be.$$
\end{proof}

\subsection{An expression for $\chi\left(\T\right)$}
In this section we give a formula for $\chi(\T)$ for curves $C$ on $\Fe$. Complements to rational cuspidal curves with three or more cusps on Hirzebruch surfaces can be shown to have $\lkkf=2$ (see \cite{MOEONC}), however, these open surfaces are no longer $\Q$-acyclic, so we do not expect the rigidity conjecture of Flenner and Zaidenberg to hold in this case. Indeed, we calculate the value of $\chi(\T)$ for the curves provided in Section \ref{4C}, and observe that for these curves we do not necessarily have $\chi(\T)=0$.

\begin{thm}\label{chithetafe}
For an irreducible rational cuspidal curve $C$ of type $(a,b)$ on $\Fe$ with $s$ cusps $p_j$ with respective $M$-numbers $M_j$, we have $$\chi(\T)=7-2a-2b-be+\sum_{j=1}^s M_j.$$
\end{thm}
\begin{proof}
By \cite[Proposition 2.4, p.445]{FlZa94},
$$K_V \ldot (K_V+D)=K_{\Fe} \ldot (K_{\Fe}+C)+\sum_{j=1}^s M_j.$$

\noi By Lemma \ref{lemflzadeffe}, we then get 
\begin{align*}
\chi(\T)&= K_V \ldot (K_V+D)-1\\
&= K_{\Fe} \ldot (K_{\Fe}+C)+\sum_{j=1}^s M_j-1\\
&= ((e-2)\LH-2\RH) \ldot \Bigl((a+e-2)\LH+(b-2)\RH\Bigr)-1 + \sum_{j=1}^s M_j\\
&= 7-2a-2b-be+\sum_{j=1}^s M_j.
\end{align*}
\end{proof}


With the above result in mind, we investigate $\chi(\T)$ further. Let $C$ be a rational cuspidal curve of type $(a,b)$ on $\Fe$, and let $(V,D)$ be as before. By the above, we have that
\begin{align*}
\chi(\T)  &:=   \h^0(V,\T)-\h^1(V,\T)+h^2(V,\T)\\ 
&\, =  K_V \ldot (K_V+D)-1\\
&\, =  7-2(a+b)-be+\sum_{j=1}^sM_j.
\end{align*}

Moreover, when $\lkkf \geq 0$, we see from Lemma \ref{kod0M} that $$\chi(\T) \leq 7.$$

If $\lkkf=2$, then it follows from a result by Iitaka in \cite[Theorem 6]{Iit} that $\h^0(V,\T)=0$. Then we have $$\chi(\T) = h^2(V,\T) -\h^1(V,\T).$$

In \cite[Lemma 4.1, p.219]{Tono05} Tono shows that if first, $\lkkf=2$, and second, the pair $(V,D)$ is \emph{almost minimal} (see \cite{Miyanishi}), then $K_V \ldot (K_V+D) \geq 0$. For plane curves, this result by Tono and a result by Wakabayashi in \cite{Wak} implies that for rational cuspidal curves with three or more cusps we have that $\chi(\T) \geq 0$. Similarly, a rational cuspidal curve on a Hirzebruch surface that fulfills the two prerequisites has $\chi(\T) \geq -1$. While a smiliar result to Wakabayashi's result ensures that three or more cusps implies $\lkkf=2$ \cite{MOEONC}, rationality, however, is no longer a guarantee for almost minimality (see \cite[p.98]{MOEPHD}). Therefore, for rational cuspidal curves with three or more cusps on a Hirzebruch surface, $\chi(\T)$ is not necessarily bounded below.

\hide{If $\lkkf=2$ \emph{and} $(V,D)$ is almost minimal, we can apply a lemma by Tono \cite[Lemma 4.1, p.219]{Tono05}. In this case $K_V \ldot (K_V+D) \geq 0$, hence $$-1 \leq \chi(\T) \leq 7.$$

Going back to the proof of Theorem \ref{ONCfe}, we see that for curves of genus $g$ on $\Fe$, we have that $n <2g+2$. As before, $n$ is the number of exceptional curves not in $D$ that will be contracted by the minimalization morphism. 

For rational cuspidal curves, we see that we have $n=\{0,1\}$. This means that we, in contrast to the situation on $\Po$, are not directly in the situation that the resolution of a rational curve gives an almost minimal pair $(V,D)$. Therefore, $\chi(\T)$ is not necessarily bounded below in this case.} 

For rational cuspidal curves with four cusps on Hirzebruch surfaces $\Fe$ and $\Fh$, where $e \geq 0$ and $h \in \{0,1\}$, the values of $\chi(\T)$ is given in Table \ref{tab:fourcuspchitheta}.

\begin{table}[H]
  \renewcommand\thesubtable{}
  \setlength{\extrarowheight}{2pt}
\centering
	{\noindent \begin{tabular*}{1\linewidth}{@{}c@{\extracolsep{\fill}}l@{}c@{}c}
	\tr
	{\textbf{Type}} & {\textbf{Cuspidal configuration}} &  $\mathbf{\chi(\T)}$ & {\textbf{Surface}}\\
\tr
	{$(2k+1,4)$}  & $[4_{k-1+e},2_3],[2],[2],[2]$& $1-k-e$ & $\Fe$\\
 	{$(3k+1-h,5)$}  & $[4_{2k-1+h},2_3],[2],[2],[2]$& $-1$ & $\Fh$ \\
	{$(2k+2-h,4)$} & $[3_{2k-1+h},2],[2_3],[2],[2]$ & $0$ &$\Fh$ \\
	{$(k+1-h,3)$}  & $[2_{n_1}],[2_{n_2}],[2_{n_3}],[2_{n_4}]$& $-1$ & $\Fh$\\
	{$(0,3)$}  & $[2],[2],[2],[2]$& $-1$ & $\mathbb{F}_2$\\
	\br
	\end{tabular*}}
	\caption[$\chi(\T)$ for rational cuspidal curves with four cusps on $\Fe$.]{$\chi(\T)$ for rational cuspidal curves with four cusps on $\Fe$ and $\Fh$. For the three first series, $k \geq 0$. For the fourth series, $k \geq 2$ and $\sum_{j=1}^4 n_j=2k+h$.}
	\label{tab:fourcuspchitheta}
	\end{table}


An important observation from this list is the fact that $\chi(\T) \leq 0$ for all these curves. We reformulate this observation in a conjecture (cf. \cite{Bobins, FlZa94}).

\begin{conj}
Let $C$ be a rational cuspidal curve with four or more cusps on a Hirzebruch surface $\Fe$. Then $\chi(\T) \leq 0$.
\end{conj}

\subsection{On the multiplicity}
In the following we establish a result on the multiplicities of the cusps on a rational cuspidal curve on a Hirzebruch surface. Note that this work is inspired by Orevkov (see \cite{Ore}).

Assume that $C$ is a rational cuspidal curve on a Hirzebruch surface $\Fe$. Let $p_1,\ldots,p_s$ denote the cusps of $C$, and $m_{p_1},\ldots,m_{p_s}$ their multiplicities. Renumber the cusps such that $m_{p_1}\geq m_{p_2} \geq \ldots \geq m_{p_s}$. Then for curves with $\lkkf \geq 0$ we are able to establish a lower bound on $m_{p_1}$. 

\begin{thm}
A rational cuspidal curve $C$ of type $(a,b)$ on $\Fe$, with $\lkkf \geq 0$ and $s$ cusps, must have at least one cusp $p_1$ with multiplicity $m:= m_{p_1}$ that satisfies the below inequality,
$$m > \frac{3}{2}+a+b-\frac{1}{2}\sqrt{1+20(a+b)+4(a^2+b^2)+4be(1-b)}.$$ 
\end{thm}

\begin{proof}
Using Lemma \ref{kod0M} and a lemma by Orevkov \cite[Lemma 4.1 and Corollary 4.2, pp.663--664]{Ore}, we get
\begin{align*}
2(a+b)+be&\geq \sum_{j=1}^s M_j\\
& \geq M_{1}+\sum_{j=2}^sM_{j}\\
&> \frac{\mu_{1}}{m}+m-3+\sum_{j=2}^s M_{j}\\
&=\frac{(b-1)(2a-2+be)}{m}+m-3+\sum_{j=2}^s\Bigr(M_{j}-\frac{\mu_{j}}{m}\Bigr)\\
&\geq \frac{2ab-2(a+b)+2+b^2e-be}{m}+m-3+\sum_{j=2}^s\Bigr(M_{j}-\frac{\mu_{j}}{m_{j}}\Bigr)\\
&\geq \frac{2ab-2(a+b)+2+b^2e-be}{m}+m-3.
\end{align*}
\noi This means that  
$$0 > \frac{2ab-2(a+b)+2+b^2e-be}{m}+m-3-2(a+b)-be.$$

\noi Let $$g(a,b,m)= \frac{2ab-2(a+b)+2+b^2e-be}{m}+m-3-2(a+b)-be.$$ Factoring $g$, we have that $g<0$ for  
$$m > \frac{3}{2}+a+b-\frac{1}{2}\sqrt{1+20(a+b)+4(a^2+b^2)+4be(1-b)}.$$ 
\end{proof}

\begin{cor}
A rational cuspidal curve $C$ of type $(a,b)$ with two or more cusps on a Hirzebruch surface $\Fe$ must have at least one cusp $p_1$ with multiplicity $m$ that satisfies the below inequality,
$$m > \frac{3}{2}+a+b-\frac{1}{2}\sqrt{1+20(a+b)+4(a^2+b^2)+4be(1-b)}.$$  
\end{cor}

\begin{proof}
The corollary follows directly from the result in \cite{MOEONC} on the logarithmic Kodaira dimension of the complement to a curve with two cusps.
\end{proof}

\begin{rem}
Note that this theorem will exclude some potential curves. For example, a rational cuspidal curve of type $(a,4)$ on $\mathbb{F}_1$ with two or more cusps (see \cite{MOEONC}) must have at least one cusp with multiplicity $m=3$ for any $a\geq6$. We also have that any rational cuspidal curve of type $(a,5)$ on $\mathbb{F}_1$ with two or more cusps must have at least one cusp with multiplicity $m=3$ for any $a\geq3$. Similarly, any rational cuspidal curve of type $(a,b)$ on $\mathbb{F}_1$ with two or more cusps and $b\geq 6$ must have at least one cusp with multiplicity $m=3$.
\end{rem}

\subsection{Real cuspidal curves}
It is well known that the known plane rational cuspidal curves with three cusps can be defined over $\R$ \cite{FlZa95,FlZa97,Fen99a}. That is not the case for the plane rational cuspidal quintic curve with cuspidal configuration ${[2_3],[2],[2],[2]}$ (see \cite{MOEPHD}). 

On the Hirzebruch surfaces, the question whether all cusps on real cuspidal curves can have real coordinates is still hard to answer. Recall that we call $C=\V(F)$ a real curve if the polynomial $F$ has real coefficients. However, all known curves on $\Fe$ can be constructed from curves on $\Po$. Since the birational transformations can be given as real transformations, if it is possible to arrange the curve on $\Po$ such that the preimages of the cusps have real coordinates, then the cusps will have real coordinates on the curve on the Hirzebruch surface as well. Note the possibility that this arrangement is not always attainable.

Considering the rational cuspidal curves on $\Fe$ with four cusps, we see that most of them are constructed from the plane rational cuspidal quintic with cuspidal configuration ${[2_3],[2],[2],[2]}$. Hence, we expect that the cusps on these curves can not all have real coordinates when the curve is real. Contrary to this intuition, however, there are examples of fourcuspidal curves with this property.

\begin{prp}
The series of rational cuspidal curves of type ${(k+1-h,3)}$, $k \geq 2$, with four cusps and cuspidal configuration ${[2_{n_1}],[2_{n_2}],[2_{n_3}],[2_{n_4}]}$, where the indices satisfy $\sum_{j=1}^4n_j=2k+h$, on the Hirzebruch surfaces $\Fh$, $h \in \{0,1\}$, has the property that all cusps can be given real coordinates on a real curve. 
\end{prp}

\begin{proof}
We have seen that the series of curves can be constructed using the plane rational cuspidal quartic $C$ with three cusps. Let $$y^2z^2+x^2z^2+x^2y^2-2xyz(x+y+z)$$ be a real defining polynomial of $C$. Then it is possible to find a tangent line to $C$ that intersects $C$ in three real points. For example, choose the line $T$ defined by $$\frac{2048}{125}x+\frac{2048}{27}y-\frac{1048576}{3375}z=0.$$ This line is tangent to $C$ at the point ${(\frac{64}{9}: \frac{64}{25}: 1)}$, and it intersects $C$ transversally at the points ${(16: \frac{16}{25}: 1)}$ and ${(\frac{4}{9}: 4: 1)}$. With this configuration, there exists a birational transformation from $\Po$ to $\PP$ that preserves the real coordinates of the cusps on $C$ and constructs a fourth cusp with real coordinates on the strict transform of $C$ on $\PP$. We blow up the two real points at the transversal intersections and contract the tangent line $T$, using the birational map from $\Po$ to $\PP$ (see \cite{MOEPHD}). In coordinates, the map is given by a composition of a (real) change of coordinates on $\Po$ and the map $\phi$,
$$\begin{array}{cccc}
\phi: &\Po &\dashrightarrow &\PP \\
&(x:y:z) &\mapsto& (x:y;z:x),
\end{array}$$
with inverse
$$\begin{array}{cccc}
\phi^{-1}:& \PP &\dashrightarrow &\Po \\
&(x_0:x_1;y_0:y_1) &\mapsto &(x_0y_1:x_1y_1:x_0y_0).
\end{array}$$
The strict transform of $C$ is a real curve $C'$ on $\PP$ of type $(3,3)$ and cuspidal configuration ${[2],[2],[2],[2]}$, and all the cusps have real coordinates.

On $\PP$, since the cusps $p_j$ have real coordinates, a fiber, say $L_j$, intersecting $C'$ at a cusp is real. Using the defining polynomial of $L_j$ to substitute one of the variables $x_0$ or $x_1$ in the defining polynomial of $C$ and removing the factor of $x_i^3$, we are left with a polynomial with real coefficients in $y_0$ and $y_1$ of degree $3$. This polynomial has a double real root, and one simple, hence real, root. The double root corresponds to the $y$-coordinates of the cusp $p_j$, and the simple root to the $y$-coordinates of a smooth intersection point $r_j$ of $C$ and $L_j$. Successively blowing up at any $r_j$ and contracting the corresponding $L_j$ lead to the desired series of curves. Since the points we blow up have real coordinates, the transformations preserve the real coordinates of the cusps. Hence, all the curves in the series can have four cusps with real coordinates. 
\end{proof}

An image of a real rational cuspidal curve of type $(3,3)$ with four ordinary cusps on $\PP$ is given in Figure \ref{4realcusps}. In the figure, the surface $\PP$ is embedded in $\mathbb{P}^3$ using the Segre embedding, and we have chosen a suitable affine covering of $\mathbb{P}^3$. The image is created in cooperation with Georg Muntingh using \verb+surfex+ \cite{Surfex}.

\begin{center}
\begin{figure}[H]
\centering
    \includegraphics[width=10cm]{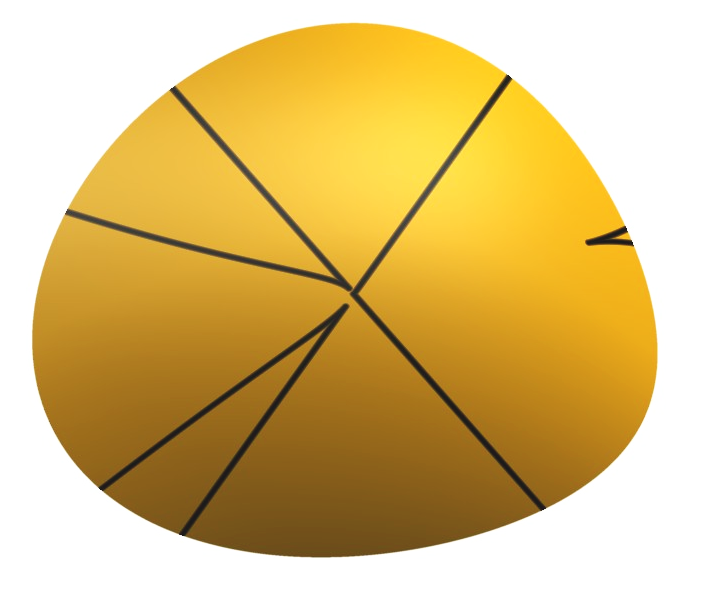}
\caption[A real rational cuspidal curve with four ordinary cusps on $\PP$.]{A real rational cuspidal curve of type $(3,3)$ with four ordinary cusps on $\PP$.}
\label{4realcusps}
\end{figure}
\end{center}

\appendix
\section{Construction of curves}\label{Appendix:A}
\noi We show by two examples the explicit construction of some of the rational cuspidal curves in Theorem \ref{thm:4cusp}. This is done with the computer program \verb+Maple+ \cite{Maple} and the package \verb+algcurves+. See \cite{MOEPHD} for description of the maps. 

\subsection{The $(1,4)$-curves of Theorem \ref{thm:4cusp}}
\begin{ex}
The curves of type $(1,4)$ on $\Fe$, $e \geq 1$, can be constructed and checked with the following code. Here we have done this up to $e=2$, but the procedure generalizes to any $e \geq 1$. We first take the defining polynomial of the plane rational cuspidal quintic with four cusps, and verify its cuspidal configuration with the command \verb+singularities+. The \verb+Maple+ output of the command \verb+singularities+ is a list with elements on the form $$[[x,y,z],m,\delta,b],$$ where $[x,y,z]$ denotes the plane projective coordinates, $m$ the multiplicity, $\delta$ the delta invariant, and $b$ the number of branches of the singularity. 
{\scriptsize
\begin{verbatim}


> with(algcurves):

> F := y^4*z-2*x*y^2*z^2+x^2*z^3+2*x^2*y^3-18*x^3*y*z-27*x^5:

> singularities(F, x, y);
{[[0, 0, 1], 2, 3, 1], [[RootOf(27*_Z^3+16*z^3), -(9/4)*RootOf(27*_Z^3+16*z^3)^2/z, 1], 2, 1, 1]}


\end{verbatim}
}

\noi We next move the curve such that we can blow up the appropriate point on $\Po$. By inspection of the defining polynomial, we find the tangent line to the curve at the cusp with multiplicity sequence $[2_3]$ to be $\V(x)$, and its smooth intersection point with the curve is $(0:1:0)$. We then change coordinates and move $(0:1:0)$ to $(0:0:1)$. 
{\scriptsize
\begin{verbatim}


> z := 1: sort(F);
-27*x^5+2*x^2*y^3-18*x^3*y+y^4-2*x*y^2+x^2

> singularities(F*x, x, y); unassign('z'):
{[[0, 0, 1], 3, 7, 2], [[0, 1, 0], 2, 1, 2],
[[(4/31)*RootOf(27*_Z^3+108*_Z-92)-24/31-(9/31)*RootOf(27*_Z^3+108*_Z-92)^2,...
 ...-24/31-(27/31)*RootOf(27*_Z^3+108*_Z-92)-(9/31)*RootOf(27*_Z^3+108*_Z-92)^2, 1], 2, 1, 1]}

> x := xa: y := za: z := ya:


\end{verbatim}
}

\noi Now we blow up $(0:0:1)$, take the strict transform and check that we have the $(1,4)$-curve on $\Fen$ by finding its singularities in all four affine coverings. Note that \verb+Maple+ provides false singularities, since the \verb+algcurve+ package considers curves as objects on $\Po$. The existing singularities have coordinates on the form $[\cdot,\cdot,1]$. 

{\scriptsize
\begin{verbatim}


> xa := x0*y1: ya := x1*y1: za := y0: 

> factor(F);
-y1*(-y0^4*x1+2*y1^2*x0*y0^2*x1^2-y1^4*x0^2*x1^3-2*y1*x0^2*y0^3+18*y1^3*x0^3*y0*x1+27*y1^4*x0^5)

> F := -y0^4*x1+2*y1^2*x0*y0^2*x1^2-y1^4*x0^2*x1^3-2*y1*x0^2*y0^3+18*y1^3*x0^3*y0*x1+27*y1^4*x0^5:

> x0 := 1: y0 := 1: singularities(F, x1, y1); unassign('x0', 'y0'):
{[[0, 1, 0], 3, 3, 3], [[1, 0, 0], 4, 9, 1], 
[[RootOf(16*_Z^3+27), -(4/9)*RootOf(16*_Z^3+27), 1], 2, 1, 1]}

> x0 := 1: y1 := 1: singularities(F, x1, y0); unassign('x0', 'y1'):
{[[1, 0, 0], 2, 3, 1], [[RootOf(16*_Z^3+27), (4/3)*RootOf(16*_Z^3+27)^2, 1], 2, 1, 1]}

> x1 := 1: y1 := 1: singularities(F, x0, y0); unassign('x1', 'y1'):
{[[0, 0, 1], 2, 3, 1], [[RootOf(27*_Z^3+16), -(9/4)*RootOf(27*_Z^3+16)^2, 1], 2, 1, 1]}

> x1 := 1: y0 := 1: singularities(F, x0, y1); unassign('x1', 'y0'):
{[[0, 1, 0], 5, 13, 4], [[1, 0, 0], 4, 12, 4], 
[[RootOf(27*_Z^3+16), (3/4)*RootOf(27*_Z^3+16), 1], 2, 1, 1]}


\end{verbatim}
}

\noi The curve on $\Fen$ is positioned in such a way that we can apply the Hirzebruch one up transformation repeatedly, and get the $(1,4)$-curve on $\Fe$ for any $e$. After one transformation we have a curve on $\mathbb{F}_2$ with the prescribed singularities. The following code verifies the latter claim.
{\scriptsize
\begin{verbatim}


> x0 := x0a: x1 := x1a: y0 := y0a: y1 := x0a*y1a: 

> F;
-y0a^4*x1a+2*x0a^3*y1a^2*y0a^2*x1a^2-x0a^6*y1a^4*x1a^3-2*x0a^3*y1a*y0a^3
+18*x0a^6*y1a^3*y0a*x1a+27*x0a^9*y1a^4

> x0a := 1: y0a := 1: singularities(F, x1a, y1a); unassign('x0a', 'y0a'):
{[[0, 1, 0], 3, 3, 3], [[1, 0, 0], 4, 9, 1], 
[[RootOf(16*_Z^3+27), -(4/9)*RootOf(16*_Z^3+27), 1], 2, 1, 1]}

> x0a := 1: y1a := 1: singularities(F, x1a, y0a); unassign('x0a', 'y1a'):
{[[1, 0, 0], 2, 3, 1], [[RootOf(16*_Z^3+27), (4/3)*RootOf(16*_Z^3+27)^2, 1], 2, 1, 1]}

> x1a := 1: y1a := 1: singularities(F, x0a, y0a); unassign('x1a', 'y1a'):
{[[0, 0, 1], 4, 9, 1], [[0, 1, 0], 5, 16, 4], [[RootOf(27*_Z^3+16), 4/3, 1], 2, 1, 1]}

> x1a := 1: y0a := 1: singularities(F, x0a, y1a); unassign('x1a', 'y0a'):
{[[0, 1, 0], 9, 45, 4], [[1, 0, 0], 4, 18, 4], [[RootOf(27*_Z^3+16), 3/4, 1], 2, 1, 1]}


\end{verbatim}
}
\end{ex}

\subsection{A $(5,4)$-curve of Theorem \ref{thm:4cusp}}
\begin{ex}
Instead of constructing the entire series of $(1,4)$ curves, we may take the $(1,4)$-curve on $\Fen$ and construct a $(5,4)$-curve on $\PP$. We then have to change coordinates before we apply the Hirzebruch one down transformation. The curve on $\Fen$ has a cusp at $(0:1;0,1)$, which we must move, say to $(0:1;1,1)$, before we apply the transformation to $\PP$. Then we check that we have the $(5,4)$-curve on $\PP$ with the prescribed singularities.
{\scriptsize
\begin{verbatim}


> x0 := x0a: x1 := x1a: y0 := y0a-x1a*y1a: y1 := y1a:

> x0a := x0b: x1a := x1b: y0a := x0*y0b: y1a := y1b:

> F;
-x1b*x0b^4*y0b^4+4*x0b^3*y0b^3*x1b^2*y1b-6*x0b^2*y0b^2*x1b^3*y1b^2
+4*x0b*y0b*x1b^4*y1b^3-x1b^5*y1b^4+2*y1b^2*x0b^3*x1b^2*y0b^2
-4*y1b^3*x0b^2*x1b^3*y0b+2*y1b^4*x0b*x1b^4+y1b^4*x0b^2*x1b^3
-2*y1b*x0b^5*y0b^3+6*y1b^2*x0b^4*y0b^2*x1b-6*y1b^3*x0b^3*y0b*x1b^2
+18*y1b^3*x0b^4*x1b*y0b-18*y1b^4*x0b^3*x1b^2+27*y1b^4*x0b^5

> x0b := 1: y0b := 1: singularities(F, x1b, y1b); unassign('x0b', 'y0b'):
{[[0, 1, 0], 5, 10, 5], [[1, 0, 0], 4, 15, 1], 
[[RootOf(16*_Z^3+27), 4/9-(16/27)*RootOf(16*_Z^3+27)+(16/81)*RootOf(16*_Z^3+27)^2, 1], 2, 1, 1]}

> x0b := 1: y1b := 1: singularities(F, x1b, y0b); unassign('x0b', 'y1b'):
{[[1, 1, 0], 2, 3, 1], 
[[RootOf(16*_Z^3+27), RootOf(16*_Z^3+27)+(4/3)*RootOf(16*_Z^3+27)^2, 1], 2, 1, 1]}

> x1b := 1: y1b := 1: singularities(F, x0b, y0b); unassign('x1b', 'y1b'):
{[[0, 1, 0], 4, 15, 1], [[1, 0, 0], 3, 3, 3], 
[[RootOf(27*_Z^3+16), -(9/4)*RootOf(27*_Z^3+16)-(27/16)*RootOf(27*_Z^3+16)^2, 1], 2, 1, 1]}

> x1b := 1: y0b := 1: singularities(F, x0b, y1b); unassign('x1b', 'y0b'):
{[[0, 0, 1], 4, 9, 1], [[0, 1, 0], 5, 10, 5], [[1, 0, 0], 4, 6, 4], 
[[RootOf(27*_Z^3+16), 4/9-(1/3)*RootOf(27*_Z^3+16)+RootOf(27*_Z^3+16)^2, 1], 2, 1, 1]}

\end{verbatim}
}
\end{ex}

\bibliographystyle{hacm}
\bibliography{bib2}

\end{document}